\documentclass[10pt]{article}

\usepackage[margin=1.15in]{geometry}
\usepackage[hidelinks]{hyperref}
\usepackage{amsmath,amsthm,amsfonts,amssymb,latexsym,mathrsfs}
\usepackage{url}

\newtheorem{thm}{Theorem}[section]
\newtheorem{prob}{Problem}[section]

\newtheorem{lemma}{Lemma}[section]
\newtheorem{cor}{Corollary}[section]

\newtheorem{exam}{Example}

\newcommand{\ms}{\mathfrak{S}}
\newcommand{\des}{{\rm des\,}}
\newcommand{\arxiv}[1]{\href{https://arxiv.org/abs/#1}{\texttt{arXiv:#1}}}

\newcommand {\dx}{\frac{\mathrm{d}}{\mathrm{d}x}}
\newcommand{\mq}{\mathcal{Q}}
\newcommand{\mqn}{\mathcal{Q}_n}
\newcommand{\plat}{{\rm plat\,}}
\newcommand{\asc}{{\rm asc\,}}
\newcommand{\exc}{{\rm exc\,}}

\newcommand{\msn}{\mathfrak{S}_n}
\newcommand{\mdn}{\mathcal{D}}

\numberwithin{equation}{section}

\title{A general method of deducing the determinantal expressions for polynomial and its derivative}

\author{
Hongyi Lou\textsuperscript{a,b}\quad
Shi-Mei Ma\textsuperscript{c}\quad
Xinzhe Song\textsuperscript{a,b}\quad
Guiying Yan\textsuperscript{a,b}\quad
Yeong-Nan Yeh\textsuperscript{d}
}

\date{}

\begin{document}

\maketitle

\begingroup
\renewcommand{\thefootnote}{\alph{footnote}}
\footnotetext[1]{Academy of Mathematics and Systems Science, Chinese Academy of Sciences, Beijing 100190, P.R. China (\href{mailto:louhongyi20@mails.ucas.ac.cn}{louhongyi20@mails.ucas.ac.cn}, \href{mailto:songxinzhe@amss.ac.cn}{songxinzhe@amss.ac.cn}, \href{mailto:yangy@amss.ac.cn}{yangy@amss.ac.cn}).}
\footnotetext[2]{University of Chinese Academy of Sciences, Beijing 100049, P.R. China.}
\footnotetext[3]{School of Mathematics and Statistics, Shandong University of Technology, Zibo 255000, Shandong, China (\href{mailto:shimeimapapers@163.com}{shimeimapapers@163.com}).}
\footnotetext[4]{College of Mathematics and Physics, Wenzhou University, Wenzhou 325035, P.R. China (\href{mailto:mayeh@alum.sinica.edu.tw}{mayeh@alum.sinica.edu.tw}).}
\endgroup

\begin{abstract}
In this paper, we present a general method of deducing the determinantal expressions for a polynomial and its derivative.
As illustrations, we provide three determinantal expressions for the derivative of the Eulerian polynomial. 
Using a functional equation discovered by Gessel,
we also establish the determinantal expressions for the second-order Eulerian polynomial and its derivative.
\end{abstract}



\section{Introduction}
%%%%%%%%%%%%%%%%%%%%%%%%%%%%%%%%%%%%%%%%%%%
%%%%%%%%%%%%%%%%%%%%%%%%%%%%%%%%%%%%%%%%%%%%%%%%%%%%%%%%%%%%%%%%%%%%%%%%%%%%%%%%%
%%%%%%%%%%%%%%%%%%%%%%%%%%%5%%%%
Following Hwang, Chern and Duh~\cite{Hwang2020}, the general Eulerian recurrence can be defined by
\begin{equation}\label{Eulerian-type recurrence}
\mathcal{P}_n(x)=(\alpha(x)n+\gamma(x))\mathcal{P}_{n-1}(x)+\beta(x)(1-x)\dx \mathcal{P}_{n-1}(x),    
\end{equation}
where $\mathcal{P}_0(x),\alpha(x),\gamma(x),\beta(x)$ are given functions of $x$.
In particular, the classical Eulerian polynomial $A_n(x)$ 
can be defined by the recursion
\begin{equation*}\label{Anx-recu}
A_n(x)=nxA_{n-1}(x)+x(1-x)\dx A_{n-1}(x),~A_0(x)=1.
\end{equation*}
It is well known that $A_n(x)$ is the descent enumerator over all permutations of $[n]=\{1,2,\ldots,n\}$.
As noted by Gessel~\cite[A065826]{Sloane}, the polynomial $\dx A_{n}(x)$ is the descent enumerator over all permutations of $[n+1]$ that start with an ascent. 

Polynomial and its derivative are closely related and they share several similar properties, 
including the distribution of zeros~\cite{Kumar2022,Liu07} and inequalities~\cite{Aziz88,Ma24}. For example, Kumar~\cite{Kumar2022} recently proved
a sharpened form of the well-known Erd\"os-Lax inequality.
This paper is motivated by the following problem.
\begin{prob}\label{problem0}
For any sequence of polynomials $\{\mathcal{P}_n(x)\}_{n\geqslant0}$ that satisfies~\eqref{Eulerian-type recurrence}, 
whether there exist two functions $f_n(x_0,x_1,x_2,\ldots,x_{n-1})$ and $g_n(x_0,x_1,x_2,\ldots,x_{n-1})$ such that 
 $$\mathcal{P}_n(x)=f_n\left(\mathcal{P}_0(x),\mathcal{P}_1(x),\mathcal{P}_2(x),\ldots,\mathcal{P}_{n-1}(x)\right),$$
  $$\dx \mathcal{P}_n(x)=g_n\left(\mathcal{P}_0(x),\mathcal{P}_1(x),\mathcal{P}_2(x),\ldots,\mathcal{P}_{n-1}(x)\right).$$
\end{prob}

Recently, determinantal representations of special numbers and polynomials have been
pursued by several authors, see~\cite{Chow2401,Chow2402,Ma2026deter,Munarini2020} and~\cite{Qi2015,QiChapman2016,QiWangGuo2016,Qi2019}.
For instance, by using a recurrence relation, 
Munarini~\cite{Munarini2020} showed that the major index polynomial over the set of derangements
can be expressed as the determinant of a Hessenberg matrix.
In particular, Qi, Wang and Guo~\cite{QiWangGuo2016} applied determinant techniques to recover two determinantal representations for derangement numbers, while Qi~\cite{Qi2019} further obtained determinantal expressions and recurrence relations for Fubini and Eulerian polynomials.
By using the $n$th derivative of a quotient of 
two functions~\cite[p.~40]{Bourbaki04}, Chow~\cite{Chow2023} found that $A_n(x)$ can be expressed as the following lower
Hessenberg determinant:
 \begin{equation}\label{Anx-det01}
  A_{n}(x)=\begin{vmatrix}
x&-1&\cdots&0&0\\
x(1-x)&2x&\cdots&0&0\\
\vdots&\vdots&&\ddots&\vdots\\
x(1-x)^{n-2}&\binom{n-1}{1}x(1-x)^{n-3}&\cdots&\binom{n-1}{n-2}x&-1\\
x(1-x)^{n-1}&\binom{n}{1}x(1-x)^{n-2}&\cdots&\binom{n}{n-2}x(1-x)&\binom{n}{n-1}x
    \end{vmatrix}_{n\times n}.
  \end{equation}
According to~\cite[Theorem~5]{Ma2026deter}, another expression of $A_n(x)$ is given as follows:
\begin{equation}\label{Anx-det02}
A_{n}(x)=\begin{vmatrix}
x&-1&0&\cdots&0\\
A_1(x)&x&-1&\cdots&0\\
A_2(x)&x(1-x)&2x&\cdots&0\\
\vdots&\vdots&\vdots&&-1\\
A_{n-1}(x)&x(1-x)^{n-2}&\binom{n-1}{1}x(1-x)^{n-3}&\cdots&\binom{n-1}{n-2}x\\
    \end{vmatrix}_{n\times n},
  \end{equation}
which was established with the aid of context-free grammars.

In the next section, we provide a general method of deducing the determinantal expressions for polynomial and its derivative, which gives a partial answer to Problem~\ref{problem0}. 
As illustrations, we provide three determinantal expressions for the derivative of Eulerian polynomial.
In Section~\ref{section03}, by using a functional equation discovered by Gessel~\cite{Gessel20},
we present the determinantal expressions for the second-order Eulerian polynomial and its derivative. 
%%%%%%%%%%%%%%%%%%%%%%%%%%%%%%%%%%%
%%%%%%%%%%%%%%%%%%%%%%%%%%%%%%%%%%%
\section{Functional equations and Eulerian polynomials}
%%%%%%%%%%%%%%%%%%%%%%%%%%%%%%%%%%%
%%%%%%%%%%%%%%%%%%%%%%%%%%%%%%%%%%%
%%%%%%%%%%%%%%%%%%%%%%%%%%%%%%%%%%%
\subsection{Preliminary}
%%%%%%%%%%%%%%%%%%%%%%%%%%%%%%%%%%%
%%%%%%%%%%%%%%%%%%%%%%%%%%%%%%%%%%%
%%%%%%%%%%%%%%%%%%%%%%%%%%%%%%%%%%%%%%%%%%%%%%%%%
\hspace*{\parindent}

The Hessenberg matrix appears frequently in the Schur decomposition for the nonsymmetric eigenvalue problem, and the factorization problem of matrices (see~\cite{Maroulas2016}).
We say that a matrix \(H_n=(h_{ij})_{1\leqslant i,j\leqslant n}\) is called a {\it lower Hessenberg matrix} if \(h_{ij}=0\) whenever \(j>i+1\).
We set \(H_0=1\). Cahill et al.~\cite{Cahill2002} found that:
\begin{equation*}\label{eq:hessenberg-general-recurrence}
\det H_n
=h_{n,n}\det H_{n-1}
+\sum_{r=1}^{n-1}
\left((-1)^{n-r}h_{n,r}\prod_{j=r}^{n-1}h_{j,j+1}\det H_{r-1}\right).
\end{equation*}
In particular, when the superdiagonal entries are all equal to \(-1\), it reduces to
\begin{equation}\label{eq:hessenberg-minus-one-recurrence}
\det H_n=\sum_{r=1}^{n}h_{n,r}\det H_{r-1}.
\end{equation}

The following result is stated in~\cite[p.~40]{Bourbaki04}. See~\cite[Lemma~6]{Chow2401} for a proof of it.
\begin{lemma}\label{lemma1}
Let $u=u(x)$ and $v=v(x)$ be two real functions which are $n$ times differentiable on an
interval $I\subset \mathbb{R}$. If one puts $\frac{\mathrm{d}^n}{\mathrm{d}x^n}\left(\frac{u}{v}\right)=(-1)^n\frac{w_n}{v^{n+1}}$ at every point where $v(x)\neq0$, then
 \begin{equation}\label{eq01}
  w_n=\begin{vmatrix}
u&v&0&0&\cdots&0\\
u'&v'&\binom{1}{1}v&0&\cdots&0\\
u''&v''&\binom{2}{1}v'&\binom{2}{2}v&\cdots&0\\
\vdots&\vdots&\vdots&\vdots&\ddots&\vdots\\
u^{(n-1)}&v^{(n-1)}&\binom{n-1}{1}v^{(n-2)}&\binom{n-1}{2}v^{(n-3)}&\cdots&\binom{n-1}{n-1}v\\
u^{(n)}&v^{(n)}&\binom{n}{1}v^{(n-1)}&\binom{n}{2}v^{(n-2)}&\cdots&\binom{n}{n-1}v{'}
    \end{vmatrix}.
  \end{equation}
\end{lemma}
The Lemma~\ref{lemma1} has been used extensively in recent years, see~\cite{Ma2026deter} for instance. 
%%%%%%%%%%%%%%%%%%%%%%%%%%%%%%%%%%%
\subsection{Main results}
%%%%%%%%%%%%%%%%%%%%%%%%%%%%%%%%%%%
%%%%%%%%%%%%%%%%%%%%%%%%%%%%%%%%%%%
%%%%%%%%%%%%%%%%%%%%%%%%%%%%%%%%%%%%%%%%%%%%%%%%%
\hspace*{\parindent}

Let \((\rho_n)_{n\geqslant 0}\) be a sequence of non-zero numbers with
\(\rho_0=1\). Let \((K_m(x))_{m\geqslant 0}\) be a sequence of polynomials. As a generalization of~\eqref{eq01}, we define the weighted Hessenberg matrix
\(H_n(x)=(h_{ij})_{1\leqslant i,j\leqslant n}\) by
\[
h_{ij}=
\begin{cases}
\displaystyle
\frac{\rho_i}{\rho_{j-1}\rho_{i-j+1}}K_{i-j}(x),& 1\leqslant j\leqslant i,\\[8pt]
-1,& j=i+1,\\
0,& j>i+1.
\end{cases}
\]
Thus we have
\begin{equation}\label{HnxK01}
H_n(x)=
\begin{pmatrix}
K_0(x) & -1 & 0 & \cdots & 0\\
K_1(x) & \dfrac{\rho_2}{\rho_1\rho_1}K_0(x) & -1 & \cdots & 0\\
K_2(x) & \dfrac{\rho_3}{\rho_1\rho_2}K_1(x)
& \dfrac{\rho_3}{\rho_2\rho_1}K_0(x)& \cdots & 0\\
\vdots & \vdots & \vdots & \ddots & -1\\
K_{n-1}(x)& \dfrac{\rho_n}{\rho_1\rho_{n-1}}K_{n-2}(x)
& \dfrac{\rho_n}{\rho_2\rho_{n-2}}K_{n-3}(x)& \cdots
& \dfrac{\rho_n}{\rho_{n-1}\rho_1}K_0(x)
\end{pmatrix}.
\end{equation}
\begin{lemma}\label{deter for primal poly}
Let \(F_n(x)=\det H_n(x)\) for \(n\geqslant 1\), where $H_n(x)$ is defined by~\eqref{HnxK01}. Let
\(
F(x;z)=1+\sum_{n\geqslant 1}F_n(x)\frac{z^n}{\rho_n}.
\)
Then
\[
F(x;z)= \frac{1}{1-K(x;z)},
\]
where $K(x;z)$ is given by $
K(x;z)=\sum_{m\geqslant0}K_m(x)\frac{z^{m+1}}{\rho_{m+1}}$.
\end{lemma}
\begin{proof}
It follows from~\eqref{eq:hessenberg-minus-one-recurrence} that
\[
F_n(x)=\sum_{r=1}^{n}
\frac{\rho_n}{\rho_{r-1}\rho_{n-r+1}}K_{n-r}(x)F_{r-1}(x)=\sum_{k=0}^{n-1}
\frac{\rho_n}{\rho_k\rho_{n-k}}K_{n-k-1}(x)F_k(x).
\]
Multiplying both sides by \(z^n/\rho_n\) and summing over \(n\geqslant1\), we obtain
\[
F(x;z)-1=K(x;z)F(x;z),
\]
and the desired result follows.
\end{proof}

The exponential generating function of $A_n(x)$ is given as follows (see~\cite{FoataSchutzenberger1970}): 
\begin{equation}\label{Axz}
A(x;z):=\sum_{n\geqslant 0}A_n(x)\frac{z^n}{n!}=\frac{1-x}{1-x{\mathrm{e}}^{(1-x)z}}=1+xz+(x+x^2)\frac{z^2}{2!}+\cdots.
\end{equation}
\begin{cor}\label{Cor01}
Setting \(A(x;z)=1/\left(1-K(x;z)\right)\), we have
\[
K(x;z)=1-\frac{1}{A(x;z)}
=
1-\frac{1-x{\mathrm{e}}^{(1-x)z}}{1-x}
=
\sum_{m\geqslant 0}x(1-x)^m\frac{z^{m+1}}{(m+1)!}.
\]
Thus \(K_m(x)=x(1-x)^m\). By Lemma~\ref{deter for primal poly}, we immediately get~\eqref{Anx-det01}.
\end{cor}

We say that $\pi\in\msn$ is a derangement if $\pi_i\neq i$ for all $i\in [n]$. 
Denote by $\mdn_n$ the set of all derangements in $\msn$. 
The {\it derangement polynomials} are defined by $$d_n(q)=\sum_{\pi\in\mdn_n}q^{\exc(\pi)},$$
where $\exc(\pi)=\#\{i\in[n-1]:~\pi_i>i\}$ is the number of excedances of $\pi$.
The generating function of $d_n(q)$ is given as follows (see~\cite[Proposition~6]{Brenti90}):
\begin{equation*}\label{dxz-EGF}
d(q;z)=\sum_{n=0}^{\infty}d_n(q)\frac{z^n}{n!}=\frac{1-q}{\mathrm{e}^{qz}-q\mathrm{e}^{z}}.
\end{equation*}

As usual, set $[0]_q=0$.
For any positive integer $m$, let $[m]_q=1+q+\cdots+q^{m-1}$.
Setting \(d(q;z)=1/\left(1-K(q;z)\right)\), we have
\[
K(q;z)=1-\frac{1}{d(q;z)}
=
1-\frac{\mathrm{e}^{qz}-q\mathrm{e}^{z}}{1-q}
=
\sum_{m\geqslant 1}q[m]_{q}\frac{z^{m+1}}{(m+1)!}.
\]
Thus $K_m(q)=q[m]_{q}$ for $m\geqslant 0$. By Lemma~\ref{deter for primal poly}, we get the following.
\begin{cor}
For $n\geqslant 1$, we have
\begin{equation*}\label{HnxK}
d_n(q)=
\begin{vmatrix}
0 & -1 & 0 & \cdots & 0\\
q & 0 & -1 & \cdots & 0\\
q[2]_q &\binom{3}{1}q
& 0& \cdots & 0\\
\vdots & \vdots & \vdots & \ddots & -1\\
q[n-1]_q& \binom{n}{1}q[n-2]_q
&  \binom{n}{2}q[n-3]_q& \cdots
& 0
\end{vmatrix}_{n\times n}.
\end{equation*}
\end{cor}

Let $F(x;z)$ be given by Lemma~\ref{deter for primal poly}.
Define $$L(x;z):=\frac{\partial}{\partial x}\ln F(x;z)=\sum_{n\geqslant 1}L_n(x)\frac{z^n}{\rho_n}.$$
We can now conclude the main result of this paper.
\begin{thm}\label{thm01}
Let \(F_n(x)=\det H_n(x)\) for \(n\geqslant 1\), where $H_n(x)$ is defined by~\eqref{HnxK01}.
We have
\begin{equation}\label{HnxK02}
\dx F_n(x)=\det \widetilde{H}_n(x),
\end{equation}
where $\widetilde{H}_n(x)$ is given as follows:
\begin{equation*}%\label{LK01}
\widetilde{H}_n(x)=
\begin{pmatrix}
L_1(x) & -1 & 0 & \cdots & 0\\
L_2(x) & \dfrac{\rho_2}{\rho_1\rho_1}K_0(x) & -1 & \cdots & 0\\
L_3(x) & \dfrac{\rho_3}{\rho_1\rho_2}K_1(x)
& \dfrac{\rho_3}{\rho_2\rho_1}K_0(x)& \cdots & 0\\
\vdots & \vdots & \vdots & \ddots & -1\\
L_{n}(x)& \dfrac{\rho_n}{\rho_1\rho_{n-1}}K_{n-2}(x)
& \dfrac{\rho_n}{\rho_2\rho_{n-2}}K_{n-3}(x)& \cdots
& \dfrac{\rho_n}{\rho_{n-1}\rho_1}K_0(x)
\end{pmatrix},
\end{equation*}
and $L_n(x)$ can be computed by the following determinantal expression:
\begin{equation}\label{LK01}
L_n(x)=
\left|
\begin{array}{ccccc}
\dx K_0(x) & -1 & 0 & \cdots & 0\\
\dx K_1(x) & \dfrac{\rho_2}{\rho_1\rho_1}K_0(x) & -1 & \cdots & 0\\
\dx K_2(x) & \dfrac{\rho_3}{\rho_1\rho_2}K_1(x) &  \dfrac{\rho_3}{\rho_2\rho_1}K_0(x) & \cdots & 0\\
\vdots  & \vdots & \vdots & \ddots & -1\\
\dx K_{n-1}(x) &
\dfrac{\rho_n}{\rho_1\rho_{n-1}}K_{n-2}(x) &
\dfrac{\rho_n}{\rho_2\rho_{n-2}}K_{n-3}(x) &
\cdots &
 \dfrac{\rho_n}{\rho_{n-1}\rho_1}K_0(x)
\end{array}
\right|,
\end{equation}
which is equivalent to
\begin{equation}\label{LK02}
        L_n(x)
        =\dx K_{n-1}(x)
        +\sum_{k=1}^{n-1}
       \frac{\rho_n}{\rho_k\rho_{n-k}}K_{n-k-1}(x)L_k(x),
\end{equation}
\end{thm}
\begin{proof}
\quad (A)
Note that 
$$L(x;z)=\frac{\partial}{\partial x}\ln F(x;z)=\frac{\frac{\partial}{\partial x}F(x;z)}{F(x;z)}.$$
It follows from Lemma~\ref{deter for primal poly} that 
\(F(x;z)=1/\left(1-K(x;z)\right)\). Thus we have
\begin{equation*}
\frac{\partial}{\partial x}F(x;z)=L(x;z)+K(x;z)\frac{\partial}{\partial x}F(x;z).
\end{equation*}
Taking the coefficients of $\frac{z^n}{\rho_n}$ leads to
\begin{equation}
\dx F_n(x)=L_n(x)+\sum_{k=1}^{n-1}\frac{\rho_n}{\rho_k\rho_{n-k}}K_{n-k-1}(x)\dx F_k(x).
\end{equation}

Expanding the determinant of $\widetilde{H}_n(x)$ along the last row, we obtain
\[
\det \widetilde H_n(x)=L_n(x)+\sum_{k=1}^{n-1}
\frac{\rho_n}{\rho_k\rho_{n-k}}K_{n-k-1}(x)\det \widetilde H_k(x).
\]
Since $\dx F_1(x)=\det \widetilde H_1(x)=L_1(x)$, we see that \(\det \widetilde H_n(x)\) and
\(\dx F_n(x)\) satisfy the same recursion and initial value, so they agree.

\quad (B)
Since \(F(x;z)=1/\left(1-K(x;z)\right)\), it follows that
\[
        \ln F(x;z)=-\ln\left(1-K(x;z)\right).
\]
Differentiating both sides with respect to \(x\) gives
\[
        \frac{\partial}{\partial x} \ln F(x;z)
        =\frac{\frac{\partial}{\partial x} K(x;z)}{1-K(x;z)}.
\]
Therefore, we arrive at
\begin{equation}\label{LK}
        L(x;z)=\frac{\frac{\partial}{\partial x} K(x;z)}{1-K(x;z)}.
\end{equation}
Equivalently,
\[
       L(x;z)=\frac{\partial}{\partial x} K(x;z)+K(x;z)L(x;z).
\]
Taking the coefficients of $\frac{z^n}{\rho_n}$ leads to~\eqref{LK02}. This completes the proof.
\end{proof}

Combining Corollary~\ref{Cor01} and Theorem~\ref{thm01}, we now give the following result.
\begin{cor}
For $n\geqslant 2$, we have
\begin{equation}\label{Anx-deri01}
\dx A_n(x)=
\begin{vmatrix}
1 & -1 & 0 & \cdots & 0\\
\dx A_1(x) & 2x & -1 & \cdots & 0\\
\dx A_2(x) & 3x(1-x) & 3x & \cdots & 0\\
\vdots & \vdots & \vdots & \ddots & -1\\
\dx A_{n-1}(x)& nx(1-x)^{n-2} & \binom n2x(1-x)^{n-3} & \cdots & nx
\end{vmatrix},
\end{equation}
\begin{equation}\label{Anx-deri02}
\dx A_{n-1}(x)=
\begin{vmatrix}
1 & -1 & 0 & \cdots & 0\\
1-2x & 2x & -1 & \cdots & 0\\
(1-x)(1-3x) & 3x(1-x) & 3x & \cdots & 0\\
\vdots & \vdots & \vdots & \ddots & -1\\
(1-x)^{n-2}(1-nx) & \binom n1 x(1-x)^{n-2} & \binom n2 x(1-x)^{n-3}
&\cdots &\binom n{n-1}x
\end{vmatrix}.
\end{equation}
\end{cor}
\begin{proof}
Let $A(x;z)$ be given by~\eqref{Axz}.
Setting
$$L(x;z)=\frac{\partial}{\partial x} \ln A(x;z) = -\frac1{1-x} + \frac{{\mathrm{e}}^{(1-x)z}(1-xz)}{1-x{\mathrm{e}}^{(1-x)z}}.$$
It is easy to verify that $L_1(x)=1$ and for $n\geqslant 2$, we have
$$L_{n}(x)=\frac{A_n(x)-nxA_{n-1}(x)}{x(1-x)}=\dx A_{n-1}(x).$$
Combining~\eqref{Anx-det01} and~\eqref{HnxK02}, we obtain~\eqref{Anx-deri01}. Moreover, by~\eqref{LK01}, we get~\eqref{Anx-deri02}.
\end{proof}

In the rest of this section, 
we deduce another determinantal expression for the derivative of Eulerian polynomial.
Let \(\mathfrak S_n\) be the set of permutations of \([n]\). For
\(\pi=\pi_1\pi_2\cdots\pi_n\in\mathfrak S_n\), we always set \(\pi_{0}=\pi_{n+1}=0\).
A {\it descent} of $\pi$ is an entry $\pi_i$ such that $\pi_i>\pi_{i+1}$. Let $\des(\pi)$ be the number of descents of $\pi$. 
The {\it Eulerian polynomials} $A_n(x)$ can be defined by
$$A_n(x)=\sum_{\pi\in\ms_n}x^{\des(\pi)}.$$
Note that 
$$\dx A_n(x)=\sum_{\pi\in\ms_n}\des(\pi)x^{\des(\pi)-1},$$
which implies that the derivative of $A_n(x)$ amounts to choosing one of the descents of $\pi\in\ms_n$. 
Let $\mathfrak D_{n+1}=\left\{\pi\in\mathfrak S_{n+1}: \pi_j=n+1
\text{ and } \pi_{j-1}>\pi_{j+1}\right\}$.
If we insert $n+1$ right after a given descent of $\pi\in\ms_n$, then this descent of $\pi$ 
is replaced by a descent of a permutation in $\ms_{n+1}$, and the number of descents is not changed under this operation, which yields that 
\begin{equation}\label{Anx-com}
\dx A_n(x)=\sum_{\pi\in \mathfrak D_{n+1}}x^{\operatorname{des}(\pi)-1}.
\end{equation}

\begin{exam}
When $n=3$, the polynomial $\dx A_3(x)$ can be seen as the descent enumerator over the following permutations:
$$1234,~1324,~1342,~2134,~2413,~2314,~2341,~3124,~3412,~3214,~3241,~3421.$$
\end{exam}

For any word $w$ over $[n]$, we say that the {\it reduced form} of $w$, written as $\rm {red}(w)$, is the word obtained by replacing 
each of the occurrences of the $i$-th smallest entry in $w$ by the number $i$.   
\begin{lemma}\label{dxAnx}
For \(n\geqslant 1\), we have
$$
\dx A_n(x) = A_{n-1}(x) + (1+x)\dx A_{n-1}(x) + 2\sum_{k=1}^{n-2} \binom{n-1}{k} A_{n-1-k}(x) \dx A_k(x).
$$
\end{lemma}
\begin{proof}
Let \(\pi\in\mathfrak  S_{n}\). Consider a decomposition of  \(\pi\) : \(\pi=\alpha\,1\,\beta\), where \(\alpha\) and \(\beta\) are possibly empty words whose letters form a partition of $\{2,3,\ldots,n\}$. 
We use $\hat \alpha ,\hat\beta$ to denote the obtained words by inserting $n+1$ into $\alpha$ or $\beta$, respectively. We distinguish three cases:
\begin{itemize}
      \item [$(i)$] When \(\alpha=\emptyset\), then \(\pi=1\beta\). If $n+1$ is inserted right after a descent of $\pi$, then 
      $\rm{red}(\hat\beta)$ is an element in $\mathfrak  D_{n}$, which contributes to the term \(\dx A_{n-1}(x)\).
    \item [$(ii)$] When \(\beta=\emptyset\), then $\pi = \alpha 1$. If $n+1$ is inserted right after a descent, 
    then there are two cases, i.e., $\hat\alpha 1$ and $\alpha 1 (n+1)$. 
    The case $\hat\alpha 1$ contributes to the term $x\dx A_{n-1}(x)$ and the case $\alpha 1 (n+1)$ contributes to the term $A_{n-1}(x).$
    \item [$(iii)$] When $\alpha\neq \emptyset$ and $\beta \neq \emptyset$, assume that $\#\alpha=k$, then there are $\binom{n-1}{k}$ possible choices for $(\alpha,\beta)$, where \(1\leqslant k\leqslant n-2\). For a fixed $(\alpha,\beta)$, since $n+1$ should be inserted right after a descent, we see that $(\hat \alpha,\beta)$ contributes to the term $A_{n-1-k}(x) \dx A_k(x)$, while $(\alpha,\hat \beta)$ contributes to the term $ A_k(x) \dx A_{n-1-k}(x)$, respectively. Summing over all \(1\leqslant k\leqslant n-2\), we get the term $2\sum_{k=1}^{n-2} \binom{n-1}{k} A_{n-1-k}(x)\dx A_k(x)$.

\end{itemize}
The aforementioned three cases exhaust all the possibilities, and this completes the proof.
\end{proof}

Combining~\eqref{eq:hessenberg-minus-one-recurrence} and Lemma~\ref{dxAnx}, we obtain the following result.
\begin{thm}
For \(n\geqslant 1\), we have
\[
\dx A_n(x)=
\begin{vmatrix}
1 & -1 & 0 & 0 & \cdots & 0\\
A_1(x) & 1+x & -1 & 0 & \cdots & 0\\
A_2(x) & 2\binom{2}{1}A_1(x) & 1+x & -1 & \cdots & 0\\
A_3(x) & 2\binom{3}{1}A_2(x) & 2\binom{3}{2}A_1(x) & 1+x & \cdots & 0\\
\vdots & \vdots & \vdots & \vdots & \ddots & -1\\
A_{n-1}(x) &
2\binom{n-1}{1}A_{n-2}(x) &
2\binom{n-1}{2}A_{n-3}(x) &
2\binom{n-1}{3}A_{n-4}(x) &
\cdots &
1+x
\end{vmatrix}.
\]
\end{thm}

\begin{exam}
When $n=4$, we have
\[
\begin{vmatrix}
1 & -1 & 0 & 0 \\
x & 1+x & -1 & 0 \\
x+x^2 & 4x & 1+x & -1 \\
x+4x^2+x^3& 6(x+x^2) & 6x & 1+x 
\end{vmatrix}=1+22x+33x^2+4x^3.
\]
\end{exam}
%%%%%%%%%%%%%%%%%%%%%%%%%%%%%%%%%%%
%%%%%%%%%%%%%%%%%%%%%%%%%%%%%%%%%%%
\section{Second-order Eulerian polynomials}\label{section03}
%%%%%%%%%%%%%%%%%%%%%%%%%%%%%%%%%%%
%%%%%%%%%%%%%%%%%%%%%%%%%%%%%%%%%%%
Let $[n]_2:=\{1,1,2,2,\ldots,n,n\}$. 
Following Gessel and Stanley~\cite{GesselStanley1978},
we say that a multipermutation $\sigma$ of $[n]_2$ is a {\it Stirling permutation} if for each $i$, $1\leqslant i\leqslant n$, 
all letters occurring between the two occurrences of $i$ are at least $i$.
Denote by $\mq_n$ the set of Stirling permutations of $[n]_2$. For example, $\mq_2=\{1122,1221,2211\}$.
For $\sigma\in\mq_n$, any entry $\sigma_{i}$  is called an {\it ascent} (resp.~{\it descent},~{\it plateau}) if $\sigma_{i}<\sigma_{i+1}$ (resp.~$\sigma_{i}>\sigma_{i+1}$,~$\sigma_{i}=\sigma_{i+1}$), where $i\in \{0,1,2,\ldots 2n\}$ and we set $\sigma_0=\sigma_{2n+1}=0$.
Let $\asc(\sigma)$ (resp.~$\des(\sigma)$,~$\plat(\sigma)$) be the number of ascents (resp.~descents, plateaux) of $\sigma$.
It is now well known that
$$C_n(x)=\sum_{\sigma\in\mqn}x^{\des(\sigma)}=\sum_{\sigma\in\mqn}x^{\asc(\sigma)}=\sum_{\sigma\in\mqn}x^{\plat(\sigma)}.$$
The second-order Eulerian polynomials $C_n(x)$
satisfy the recurrence relation
\[
C_n(x)=(2n-1)xC_{n-1}(x)+x(1-x)\dx C_{n-1}(x),~C_0(x)=1.
\]
In particular, 
$C_1(x)=x,~C_2(x)=x+2x^2,~C_3(x)=x+8x^2+6x^3$.

There is an increasing interest in properties of enumerative polynomials over Stirling permutations, see~\cite{Haglund12,Liu23,Ma2026JCTA}.
For example, Haglund and Visontai~\cite{Haglund12} investigated the stability of multivariate Eulerian-type polynomials over
generalized Stirling permutations.
Since each Stirling permutation $\sigma\in\mqn$ can be represented as
$\sigma'1\sigma''1\sigma'''$, where $\sigma'$, $\sigma''$ and $\sigma'''$ are all Stirling permutations (possibly empty).
Clearly, the descent number of $\sigma$ is the sum of the descent numbers of $\sigma'$, $\sigma''$ and $\sigma'''$, 
unless $\sigma'''$ is empty in which case $\sigma$ has an additional descent. By this observation,
Gessel~\cite{Gessel20} found that
\begin{equation}\label{Gessel}
\frac{\mathrm{\partial}}{\mathrm{\partial}z}C(x;z)=C^2(x;z)\left(C(x;z)+x-1\right),
\end{equation}
where $C(x;z)=\sum_{n=0}^\infty C_n(x)\frac{z^n}{n!}$. 
As an application of~\eqref{Gessel}, we now give an enlightening proof of the following result, which was recently observed
from a recursion~\cite[Theorem 3]{Ma2026JCTA}.
\begin{thm}\label{thm:second-order-kernel}
The second-order Eulerian polynomial \(C_n(x)\) can be expressed as the following lower Hessenberg determinant of order \(n\):
\begin{equation}\label{Cnx-det}
C_{n}(x)=\begin{vmatrix}
x&-1&0&\cdots&0&0\\
C_1(x)&\binom{2}{1}x&-1&\cdots&0&0\\
C_2(x)&\binom{3}{1}C_1(x)&\binom{3}{2}x&\cdots&0&0\\
\vdots&\vdots&\vdots&\ddots&\ddots&\vdots\\
C_{n-2}(x)&\binom{n-1}{1}C_{n-3}(x)&\binom{n-1}{2}C_{n-4}(x)&\cdots&\binom{n-1}{n-2}x&-1\\
C_{n-1}(x)&\binom{n}{1}C_{n-2}(x)&\binom{n}{2}C_{n-3}(x)&\cdots&\binom{n}{n-2}C_1(x)&\binom{n}{n-1}x
    \end{vmatrix}_{n\times n}.
  \end{equation}
\end{thm}
\begin{proof}
For $C(x;z)=\sum_{n=0}^\infty C_n(x)\frac{z^n}{n!}$,
if we write \(C(x;z)=1/\left(1-K^C(x;z)\right)\), then $$K^C(x;z)=1-1/C(x;z).$$ 
It follows from~\eqref{Gessel} that 
$$\frac{\partial}{\partial z}K^C(x;z)=\frac{\frac{\partial}{\partial z}C(x;z)}{C^2(x;z)}=C(x;z)+x-1.$$
Therefore, we obtain
\[
K^C(x;z)=xz+\int_0^z\bigl(C(x;s)-1\bigr)\,ds
=xz+\sum_{m\geqslant 1}C_m(x)\frac{z^{m+1}}{(m+1)!}.
\]
Thus $K^C_0(x)=x$ and $K^C_m(x)=C_m(x)$ for $m\geqslant 1$.
Then~\eqref{Cnx-det} follows from Lemma~\ref{deter for primal poly}.
\end{proof}

We now define
\[
L^C(x;z):=\sum_{n\geqslant 1}L_n^{C}(x)\frac{z^n}{n!}=\frac{\partial}{\partial x}\ln C(x;z).
\]
As given in~\eqref{LK}, we have
\[
L^C(x;z)=\frac{\frac{\partial}{\partial x}K^C(x;z)}{1-K^C(x;z)}.
\]

Therefore, by Theorem~\ref{thm01}, we obtain the following result.
\begin{thm}\label{thm:second-order-same-kernel-derivative}
We have
\[
\dx C_n(x)=
\left|
\begin{array}{ccccc}
L_1^{C}(x) & -1 & 0 & \cdots & 0\\
L_2^{C}(x) & 2x & -1 & \cdots & 0\\
L_3^{C}(x) & 3C_1(x) & 3x & \cdots & 0\\
\vdots & \vdots & \vdots & \ddots & -1\\
L_n^{C}(x)
& \binom n1C_{n-2}(x)
& \binom n2C_{n-3}(x)
& \cdots
& \binom n{n-1}x
\end{array}
\right|,
\]
where $L_n^{C}(x)$ can be computed by the following determinantal expression:
\[
L_n^{C}(x)=
\left|
\begin{array}{ccccc}
1 & -1 & 0 & \cdots & 0\\
\dx C_1(x) & 2x & -1 & \cdots & 0\\
\dx C_2(x) & 3C_1(x) & 3x & \cdots & 0\\
\vdots & \vdots & \vdots & \ddots & -1\\
\dx C_{n-1}(x)
& \binom n1C_{n-2}(x)
& \binom n2C_{n-3}(x)
& \cdots
& \binom n{n-1}x
\end{array}
\right|.
\]
\end{thm}

\begin{exam}
Note that $L_1^C(x)=1$, $L_2^C(x)=1+2x$ and $L_3^C(x)=1+10x+6x^2$.
Then 
\[
\dx C_3(x)=\left|
\begin{array}{ccccc}
1 & -1 & 0\\
1+2x & 2x & -1 \\
1+10x+6x^2 & 3C_1(x) & 3x \\
\end{array}
\right|=1+16x+18x^2.
\]
\end{exam}

For $\sigma\in\mq_n$, if the two copies of $n+1$ are inserted right after a plateau of $\sigma$, then the two copies of $n+1$ must be adjacent, so we create one ascent and one descent, while the number of plateaux is not changed.
In the same way as the proof of~\eqref{Anx-com}, we find that
\[
\dx C_n(x)
=
\sum_{\sigma\in\mathfrak D_{n+1}^{(2)}}
x^{\operatorname{plat}(\sigma)-1},
\]
where $
\mathfrak D_{n+1}^{(2)}=
\left\{\sigma\in\mathcal Q_{n+1}:
\text{if }\sigma_j=\sigma_{j+1}=n+1,\text{ then }
\sigma_{j-1}=\sigma_{j+2}\right\}$. 
\begin{exam}
The polynomial $\dx C_2(x)=1+4x$ is the plateau enumerator of the following Stirling permutations:
$112332,~133122,~123321,~221331,~233211$.
\end{exam} 
%%%%%%%%%%%%%%%%%%%%%%%%%%%%%%%%%%%
%%%%%%%%%%%%%%%%%%%%%%%%%%%%%%%%%%%
%%%%%%%%%%%%%%%%%%%%%%%%%%%%%%%%%%%
%%%%%%%%%%%%%%%%%%%%%%%%%%%%%%%%%%%

\end{document}